\documentclass[12pt]{amsart}
\usepackage{geometry}
\usepackage{graphicx}
\usepackage{amssymb}
\usepackage{epstopdf}
\usepackage{amsmath,amscd}
\usepackage{amsthm}
\usepackage{enumerate}
\usepackage{url,verbatim,soul}
\usepackage{subfig}
\usepackage{enumitem}

\RequirePackage[colorlinks,citecolor=blue,urlcolor=blue]{hyperref}
\usepackage{breakurl}
\theoremstyle{plain}
\newtheorem{theorem}{Theorem}
\newtheorem{definition}[theorem]{Definition}
\newtheorem{lemma}[theorem]{Lemma}
\newtheorem{proposition}[theorem]{Proposition}
\newtheorem{corollary}[theorem]{Corollary}

\newtheorem{remark}[theorem]{Remark}

\newcommand\es{\varnothing}

\newcommand\ol{\overline}
\newcommand\Aut{\mathrm{Aut}}

\newcommand\sQ{{\mathcal Q}}
\newcommand\sH{{\mathcal H}}

\newcommand\sF{{\mathcal F}}
\newcommand\RR{{\mathbb R}}
\newcommand\ZZ{{\mathbb Z}}

\newcommand\PP{{\mathbb P}}

\newcommand\om{\omega}
\newcommand\g{\gamma}

\newcommand\si{\sigma}
\newcommand\eps{\epsilon}
\newcommand\De{\Delta}
\newcommand\qq{\qquad}
\newcommand\q{\quad}
\newcommand\resp{respectively}

\newcommand\oo{\infty}

\newcommand\sT{{\mathcal T}}

\newcommand\sL{{\mathcal L}}
\newcommand\Ga{\Gamma}
\newcommand\Om{\Omega}
\newcommand\La{\Lambda}

\newcommand\de{\delta}

\newcommand\lra{\leftrightarrow}

\newcommand\pc{p_{\text{\rm c}}}
\newcommand\pcs{\pc^{\text{\rm site}}}
\newcommand\pu{p_{\text{\rm u}}}
\newcommand\pus{\pu^{\text{\rm site}}}
\newcommand\pcb{\pc^{\text{\rm bond}}}
\newcommand\pub{\pu^{\text{\rm bond}}}

\newcommand\Md{\Phi}

\newcommand\Gd{\what G}

\newcommand\dw{d_{\Gd}}
\newcommand\dgs{d_{G_*}}

\newcommand\sias{\si_A^*}

\newcommand\wsias{\what\si_A}
\newcommand\wsims{\what\si_M}

\renewcommand\ell{l}

\newcommand\pd{\partial}

\newcommand\sm{\setminus}

\newcommand\piv{\mathrm{Pi}}
\renewcommand\div{\mathrm{Di}}
\newcommand\Pif{\what\Pi}

\newcounter{mycount}\newcounter{mycount2}\newcounter{mycount3}
\newenvironment{romlist}{\begin{list}{\rm(\roman{mycount2})}%
   {\usecounter{mycount2}\labelwidth=1cm\itemsep 0pt}}{\end{list}}

\newenvironment{letlist}{\begin{list}{\rm(\alph{mycount})}%
   {\usecounter{mycount}\labelwidth=1cm\itemsep 0pt}}{\end{list}}
\newenvironment{Alist}{\begin{list}{\MakeUppercase{\rm\alph{mycount}}.}%
   {\usecounter{mycount}\labelwidth=1cm\itemsep 0pt}}{\end{list}}
\newcounter{newcount1}

\newcommand\ga{\gamma}

\newcommand\diam{{\text{\rm diam}}}

\newcommand\what{\widehat}
\newcommand\PPps{\PP_{p,s}}

\newcommand\nst{non-self-touching}
\newcommand\nt{non-touching}
\numberwithin{equation}{section}
\numberwithin{theorem}{section}
\numberwithin{figure}{section}

\newcommand\dinst{$2\oo$-nst path}
\DeclareMathOperator*{\argmax}{\arg\!\max}

\title[Percolation critical  probabilities]{Percolation critical probabilities\\ of matching lattice-pairs}

\author{Geoffrey R.\ Grimmett}
\address{(GRG) Centre for
Mathematical Sciences, Cambridge University, Wilberforce Road,
Cambridge CB3 0WB, UK} 

\email{g.r.grimmett@statslab.cam.ac.uk}
\urladdr{\url{http://www.statslab.cam.ac.uk/~grg/}}
\address{(ZL) Department of Mathematics,
University of Connecticut, Storrs, Connecticut 06269-3009, USA} \email{zhongyang.li@uconn.edu}
\urladdr{\url{http://www.math.uconn.edu/~zhongyang/}}

\author{Zhongyang Li}
\date{2 March 2022, revised 17 February 2024}

\begin{document}

\begin{abstract}
A necessary and sufficient condition is established for the strict inequality $\pc(G_*)<\pc(G)$
between the critical probabilities of site percolation on a one-ended,
quasi-transitive, plane graph $G$ and 
on its matching graph $G_*$.
When $G$ is transitive, strict inequality holds if and only if $G$ is not a triangulation.

The basic approach is the standard method of enhancements, but its implemention has complexity
arising from  the non-Euclidean (hyperbolic) space,
 the study of site (rather than bond) percolation, and the generality
of the assumption of quasi-transitivity. 

This result is complementary to the work of the authors 
(\lq\lq Hyperbolic site percolation", {\tt arXiv:2203.00981}) on the equality $\pu(G) + \pc(G_*) = 1$,
where $\pu$ is the critical probability for the existence of a unique infinite open cluster.
It implies for transitive, one-ended
 $G$ that  $\pu(G) + \pc(G) \ge 1$, with equality if and only if $G$ is a triangulation.
\end{abstract}
\maketitle

\section{Strict inequalities for percolation probabilities}\label{eh}

It is fundamental to the percolation model on a graph $G$ that there
exists a \lq critical probability' $\pc(G)$ marking the onset of infinite open clusters. 
Two questions arise immediately.
\begin{letlist}
\item  What can be said about the value of $\pc(G)$?
\item For what values of the percolation density $p$ is there a \emph{unique} infinite cluster? 
\end{letlist}
These questions have attracted a great deal of attention since percolation was introduced by
Broadbent and Hammersley \cite{BH57} in 1957. They turn out to be more tractable when $G$ is planar.

Amongst exact calculations of $\pc(G)$, those for bond percolation on the square, triangular,
and hexagonal lattices have been especially influential (see \cite{K81, Wier81}, and also the book \cite{GP}).
Earlier discussion (falling short of rigorous proof)  of these values was provided by Sykes and Essam \cite{SE64} in 1964. 
The last paper includes also an account of site percolation on the triangular lattice, 
and a discussion of site percolation on a so-called \lq matching pair' of planar lattices. 
This term is explained in the companion paper \cite{GL21a}; 
the current work is concerned with the matching pair $(G,G_*)$,
where the so-called matching graph 
$G_*$ is defined as follows.

Let $G=(V,E)$ be a planar graph, 
embedded in the plane $\RR^2$ in such way that two edges may intersect only at their endpoints.
A \emph{face} of $G$ is a connected component of $\RR^2\sm E$. The boundary of a bounded face $F$ is comprised 
of edges of $G$. The \emph{matching graph} of $G$, denoted $G_*$, is obtained from $G$ by adding all diagonals
to all faces. See Figure \ref{fig:mg}.
Evidently, $G_*=G$ when $G$ is a triangulation. 
A graph with connectivity $1$ or $2$ may have a multiplicity of 
non-homeomorphic planar embeddings, and therefore
there is potential ambiguity over the definition of its matching and dual graphs (see Theorem \ref{p21}(c)).

\begin{remark}\label{inf-face}
A face $F$ of the above graph $G$ may be unbounded, in which case its boundary comprises infinitely many edges and vertices. Such $F$ generates an infinite complete subgraph of $G_*$,
on which a percolation process is trivial. We shall usually assume that all faces are bounded.
Since our graphs are assumed quasi-transitive, this is equivalent to assuming that 
$G$ is one-ended. (See \cite{Ho}, \cite[Prop.\ 2.1]{Bab97}.)
For quasi-transitive graphs with two or infinitely many ends, see Remark \ref{inf-face2}.   
\end{remark}

\begin{figure}
\centerline{\includegraphics[width=0.5\textwidth]{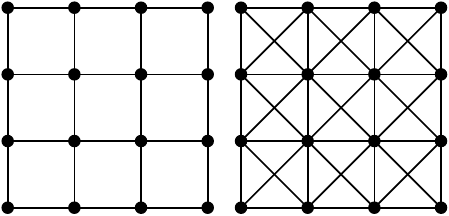}}
\caption{The square lattice $\ZZ^2$ and its matching graph.}\label{fig:mg}
\end{figure}

Sykes and Essam presented motivation for the exact relationship 
\begin{equation}\label{eq:exact}
\pcs(G)+\pcs(G_*)=1,
\end{equation}
and this has been verified in a number of cases when $G$ is amenable
(see \cite{vdB81,K81}).
Note  that,
since $G$ is a subgraph of $G_*$, it is trivial that
\begin{equation}\label{eq:exact2}
\pcs(G_*)\le \pcs(G).
\end{equation}
It is less trivial to prove strict inequality in \eqref{eq:exact2} for non-triangulations, and indeed this
sometimes fails to hold. 

Suppose that $G$ is planar, quasi-transitive, one-ended, and possibly non-amenable. 
If we are to embed $G$ in a plane in an appropriate fashion,
the plane in question may need to be hyperbolic rather than Euclidean. 
Site percolation in the hyperbolic plane
is the subject of the recent paper \cite{GL21a}, where
it is proved, amongst other things, that  
\begin{equation}\label{eq:exact3}
\pus(G) + \pcs(G_*) = 1,
\end{equation}
where $\pus$ is the critical probability for the existence of a \emph{unique} infinite open cluster.
When $G$ is amenable, we have $\pcs(G)=\pus(G)$, in agreement with
\eqref{eq:exact} 
(see \cite[Chap.\ 7]{LP} for a discussion of critical points of 
quasi-transitive, amenable graphs). 
By \eqref{eq:exact2}, we have   $\pus(G) + \pcs(G) \ge 1$, and it becomes
desirable to know when strict inequality holds. 
(When $G$ is non-amenable, it is proved in \cite{bs00} that $\pcs(G)<\pus(G)$.)  

Let $\sT$ (\resp, $\sQ$) be the set of all infinite, connected, locally finite, plane, $2$-connected,
simple graphs that are in addition
transitive (\resp, quasi-transitive). (It is explained in \cite[Rem.\ 3.4]{GL21a} that the
assumption of $2$-connectedness is innocent in the context of site percolation.)
A path $(\dots,x_{-1},x_0,x_1,\dots)$ of $G_*$ is called \emph{\nst} if, for all $i$, $j$,
two vertices $x_i$ and $x_j$ are adjacent if and only if $|i-j|=1$.  Here is the main theorem of the current work,
followed by a corollary.

\begin{theorem}\label{maintheorem}
Let $G\in\sQ$ be one-ended. Then $\pcs(G_*)<\pcs(G)$ if
and only if $G_*$ contains some doubly-infinite, \nst\ path that includes some diagonal of $G$.
\end{theorem}

Theorem \ref{maintheorem} is proved in Section \ref{sec:6} using
methods derived in Section \ref{sec:5}.

\begin{corollary}\label{cor}
Let $G\in\sQ$ be one-ended. Then
$\pus(G)+\pcs(G) \ge 1$, with strict inequality if and only if the condition of Theorem \ref{maintheorem} holds.
\end{corollary}

\begin{proof}[Proof of Corollary \ref{cor}]
The given (weak) inequality is proved at \cite[Thm 1.1(b)]{GL21a},
and the strict inequality holds by \eqref{eq:exact3} and Theorem \ref{maintheorem}.
\end{proof}

We turn to examples of Theorem \ref{maintheorem} in action. Firstly, 
the condition of the theorem is satisfied by all transitive, one-ended non-triangulations
$G\in\sT$, as in the next theorem.

\begin{theorem}\label{thm:trans0}
Let $G\in\sT$ be one-ended but not a triangulation. Then $G$
satisfes the condition of Theorem \ref{maintheorem},  and therefore $\pc(G_*)<\pc(G)$.
\end{theorem}

This is essentially the assertion of the forthcoming Theorem \ref{thm:trans},
which is proved
in Section \ref{ssec:proof72}  by the so-called metric method.
The inequality of Theorem \ref{thm:trans0} then holds by Theorem \ref{maintheorem}.

The situation for quasi-transitive graphs $G$ is  more complicated, 
and we have no useful necessary and
sufficient condition for the inequality $\pc(G_*)<\pc(G)$. 
Instead, we include in Section \ref{sec:examples}
a sufficient (but not necessary) condition. 
(\emph{Note added before publication}: the quasi-transitive case is treated in \cite{G24}.)

\begin{remark}\label{inf-face2}
The above results are subject to the assumption that $G$ is \emph{one-ended}.
By \cite{Ho} and \cite[Prop.\ 2.1]{Bab97}, the number $\eta$ of ends 
of $G\in\sQ$ lies in the set $\{1,2,\oo\}$. As in
Remark \ref{inf-face}, we have that $\pc(G_*)=0$ if $\eta\ne 1$. On the other hand,
it is standard that
$\pc(G)\ge 1/(\De-1)$ where $\De$ is the maximum vertex-degree of $G$.
The inequality $\pc(G_*)<\pc(G)$ is thus trivial when $\eta\ne 1$. 
\end{remark}

There follow some remarks about the proof of Theorem \ref{maintheorem}. 
The general approach of the proof is to use
the method of enhancements, as introduced and developed in \cite{AG} (though there is earlier work 
of relevance, including \cite{Men87}). 
While this approach is fairly standard, and the above result natural, the proof turns out 
to have substantial complexity arising from the generality of the assumptions on $G$,
and the fact that we are studying site (rather than bond) percolation (see \cite{BBR}); the proof is, in contrast,
fairly immediate for the amenable, planar lattices mentioned above.

We remark that the version of \eqref{eq:exact3} for bond percolation, namely
\begin{equation}\label{eqLbndexact3}
\pub (G)+\pcb (G^+)=1,
\end{equation}
was proved by Benjamini and Schramm \cite[Thm 3.8]{bs00} for one-ended, non-amenable, plane, transitive graphs.
Here, $G^+$ denotes the dual graph of $G$. (The amenable case is standard.) 
The basic difference between the bond and site 
problems is the following.
In the study of bond percolation, one is interested in open \emph{self-avoiding}
paths, whereas for site percolation we study open, \emph{\nst} paths ---
given an infinite path $(\dots,x_{-1},x_0,x_1,\dots)$ such that, for some 
$i+1<j$, $x_i$ 
and $x_j$ are adjacent, 
the states of vertices $x_{i+1},\dots,x_{j-1}$ are independent of the event 
that the path contains  an infinite,
open sub-path. That is, one can cut out the loop.

The central idea of the proof of Theorem \ref{maintheorem} is as follows.
Suppose $G$ satisfies the given assumptions, and write $\pi$ for the
given doubly-infinite path containing the diagonal $d$. In order to apply the enhancement method, 
one needs to show that, if $z$ is a pivotal vertex
for the existence of a long (but finite) open path of $G_*$ between given regions $A$, $B$ of space,
after making local changes to the configuration one may find a pivotal diagonal near $z$.
This is achieved by a surgery of paths. First, one cuts a finite subpath $\pi'$ from $\pi$ containing the diagonal $d$.
Then one inserts a translate of $\pi'$ into an open path $\nu$ from $A$ to $B$ in which $z$ is pivotal.
Such insertion requires \lq adjustments' near the interfaces of these two paths, 
and it must be achieved without sacrifice of the non-self-touching property.
It is an impediment to this surgery that $G_*$ is non-planar (unless $G$ is a triangulation), and thus one works
instead with a graph, denoted $\what G$, that is obtained from $G$ by placing a new vertex within 
each non-triangular face of $G$ and joining this new vertex to each vertex of the face.  

Turning to the contents of the current article, after the introductory Section \ref{sec:not}, we explain
the relevance of Theorem \ref{maintheorem} to transitive and quasi-transitive graphs in Section \ref{sec:examples}.  
The proofs begin with some preliminary observations in Section \ref{sec:5}, and the main theorem is proved
in Section \ref{sec:6}.  The claim of Section \ref{sec:examples} for quasi-transitive graphs 
is proved  in Section \ref{sec:emb2}.

\section{Notation and basic properties}\label{sec:not}

\subsection{Graph embeddings}\label{ssec:1}

We shall assume familiarity with basic graph theory and its notation, and refer the reader to \cite{GL21a} for relevant definitions. Let $\sQ$ be given as prior to Theorem \ref{maintheorem}, 
and let $\sT$ be the subset of $\sQ$ comprising the transitive graphs.

An \emph{embedding} of a graph $G=(V,E)$ (with underling $1$-complex denoted $\vert G\vert$)
in a surface $S$ is a continuous map $\phi: \vert G\vert \to S$ such that the induced map $\vert G\vert
\to \phi(\vert G\vert)$ is a homeomorphism. An embedding $\phi$ is called \emph{cellular} if 
$S\setminus \phi(G)$ is a disjoint union of spaces 
homeomorphic to an open disc. (See \cite{Moh} and \cite[Sect.\ 3.2]{MT}.)

We are concerned here with embeddings of planar graphs in either the Euclidean or hyperbolic planes, 
and \emph{we shall use $\sH$ to
denote either of these as appropriate for the setting}.
A useful summary of hyperbolic geometry may be found in
\cite{CFKP} (see also \cite{Iver}).  
An embedding of a  graph $G$  in $\sH$ is called \emph{proper} if every compact subset of $\sH$ 
contains only finitely many vertices of $G$ and intersects only finitely many edges. 
Henceforth, all embeddings will be assumed to be proper. 

An \emph{Archimedean tiling} (or \emph{uniform tiling})
of a two-dimensional Riemannian manifold 
is a tiling by regular polygons such that its isometry group (of the tiling)
acts transitively on its vertex-set.
The edges of the tiling are geodesics. 
A discussion of amenability may be found in \cite[Sect.\ 6.1]{LP}.

Some known facts concerning embeddings follow. 

\begin{theorem}\label{p21}\mbox{\hfil}
\begin{letlist}
\item \cite[Thms 3.1,  4.2]{Bab97}
If $G\in\sT$ is one-ended, then $G$ may be embedded in $\sH$
as an Archimedean tiling, and all automorphisms of $G$ extend to isometries of $\sH$.
If $G\in\sQ$ is one-ended and  $3$-connected, then $G$ may be embedded in $\sH$
such that all automorphisms of $G$ extend to isometries of $\sH$. 
Furthermore, the target space $\sH$ denotes the Euclidean plane if and only if $G$ is amenable.

\item \cite[p.\ 42]{Moh} 
Let $G$ be a $3$-connected graph, cellularly embedded in $\sH$ 
such that all faces are of finite size.
Then $G$ is uniquely embeddable in the sense that for any two cellular embeddings $\phi_1: G\to S_1$,
$\phi_2:G\to S_2$ into planar surfaces $S_1$, $S_2$, there is a homeomorphism $\tau:S_1\to S_2$ such that
$\phi_2=\tau\phi_1$.  

\item \cite[Thm 8.25 and Section 8.8]{LP} 
If $G=(V,E)\in\sQ$ is one-ended, there exists some embedding of $G$ in $\sH$
such that the edges coincide with geodesics, the dual 
graph $G^+$ is quasi-transitive, and all automorphisms of $G$ extend to isometries of $\sH$. 
Such an embedding is called \emph{canonical}.

\end{letlist}
\end{theorem}

\begin{remark}\label{rem:emb}
\mbox{\hfil}
\begin{letlist}
\item 
All one-ended, transitive, planar graphs are $3$-connected,
and all embeddings of a one-ended, quasi-transitive, planar graph have only finite faces.
\item 
By Theorem \ref{p21}(b), any one-ended $G\in\sQ$ that is in addition transitive 
has a unique  cellular embedding in $\sH$ up to homeomorphism.
Hence, the matching and dual graphs of $G$ are independent of the embedding.
\item  
The conclusion of part (b) holds for any one-ended, $3$-connected $G\in \sQ$. 
\item For a one-ended, $2$-connected  $G\in\sQ$,  we fix
a canonical embedding (in the sense of Theorem \ref{p21}(c)). With this  given, 
the dual graph $G^+$ and the matching graph $G_*$ are quasi-transitive, 
and furthermore 
the boundary  of every face is a cycle of $G$.
\end{letlist}
\end{remark}

We give a formal definition of the matching graph of a planar graph $G=(V,E)$.
Firstly, one embeds $G$ in the plane in such a way that two edges intersect only at their endpoints; such an embedded
graph is called a \emph{plane graph}. 
A \emph{face} of a plane graph $G$ is a connected component of $\sH\sm E$. 
In this work we shall treat 
only one-ended graphs, for which all faces $G$ are bounded with (topological) boundaries $\partial F$ comprised of
finitely many edges; the \emph{size} of $F$ is the number of edges in its boundary. 
A \emph{cycle} $C$ of a simple graph $G=(V,E)$ is a sequence $v_0,v_1,\dots,  v_{n+1}=v_0$ of vertices $v_i$ such that $n \ge 3$, 
$e_i:=\langle v_i,v_{i+1}\rangle$ satisfies
$e_i\in E$ for $i=0,1,\dots,n$, and $v_0,v_1,\dots,v_n$ are distinct.
Let $G$ be a plane graph, duly embedded in the Euclidean or hyperbolic plane.
In this case we write $C^\circ$ for the bounded component of $\sH\sm C$,
and $\ol C = C\cup C^\circ$ for the closure of $C^{\circ}$.

Let $V(\pd F)$ be the set of vertices lying along 
the boundary of the face $F$. For each face $F$ and each non-adjacent pair $x,y\in V(\pd  F)$,
we add an edge inside $F$ between $x$ and $y$.
We write 
$G_*=(V,E_*)$ for the ensuing \emph{matching graph} of $G$. An edge $e\in E_*\sm E$ 
is called a  \emph{diagonal} of $G$ or of $G_*$, and it is denoted $\de(a,b)$ where $a$, $b$ are its endvertices.
If $\de(a,b)$ is a diagonal,  $a$ and $b$ are called \emph{$\ast$-neighbours}.

Note that $G_*$ depends on the particular embedding of $G$.
If $G$ is $3$-connected then, by Theorem \ref{p21}(b), it has a unique embedding up to homeomorphism.
If $G$ is $2$-connected but not $3$-connected, we need to be definite about the choice 
of embedding, and we require it henceforth to be
\lq canonical' in the sense of Theorem \ref{p21}(c).  

\subsection{Further notation}\label{ssec:2}
A plane graph $G$ is called a \emph{triangulation} it every face is bounded by a $3$-cycle.
The automorphism group of the graph $G=(V,E)$ is denoted $\Aut(G)$. 
The orbit of $v\in V$ is
written $\Aut(G)v$, and we let 
\begin{equation}\label{eq:defde}
\De=\min\bigl\{k: \text{for } v,w\in V,\text{ we have } d_G(\Aut(G)v,\Aut(G)w)\le k\bigr\},
\end{equation}
where 
$$
d_G(A,B)= \min\{d_G(a,b): a\in A,\, b\in B\}, \qq A,B\subseteq V,
$$
and $d_G$ denotes graph-distance in $G$.
We write $u\sim v$ if $u,v\in V$ are adjacent, which is to say that $d_G(u,v)=1$.
For any $G$, we fix some vertex denoted $v_0$.

We shall work with one-ended graphs $G\in\sQ$.
Since $G$ is assumed one-ended and $2$-connected, all its faces 
are bounded, with boundaries which are cycles of $G$ (see Remark  \ref{rem:emb}(d)).  

\begin{definition}\label{def:nst}
A path $\pi=(\dots,x_{-1},x_0,x_1\dots)$ of a graph $H$ is
called \emph{\nst} if $d_H(x_i,x_j)\ge2$ when $\vert j-i\vert\ge2$. 
A cycle $C=(v_0,v_1,\dots,v_n, v_{n+1}=v_0)$ of $H$
is called \nst\ if $d_H(x_i,x_j)\ge 2$ whenever $|i-j|\ge 2$ (with index-arithmetic 
modulo $n+1$). 
\end{definition}

Non-self-touching paths and cycles arise naturally when studying \emph{site} percolation (such paths were
called \emph{stiff} in  \cite{AG}, and \emph{self-repelling} in \cite[p.\ 66]{GP}).

We shall consider \nst\ paths in two graphs derived from a given $G\in\sQ$, namely its
matching graph $G_*$, and the graph $\Gd$ obtained by adding a site within each 
face $F$ of size $4$ or more, and connecting every vertex of $F$ to this new site. 
The graph $G_*$ may possess parallel edges.
The property of being \nst\ is indifferent to the existence of parallel edges, since
it is given in terms of the vertex-set of $\pi$ and the adjacency relation of $H$.

Here is the fundamental property of graphs that implies strict inequality of critical points. 
This turns out to be equivalent to a more 
technical \lq local' property, as described in Section \ref{ssec:4.2}; see Theorem \ref{thm:13}. 
As a shorthand, henceforth we abbreviate \lq doubly-infinite \nst\ path' to \lq\dinst'. 

\begin{definition}\label{def10}
The graph $G\in\sQ$ is said to have property $\Pi$ if $G_*$ contains some \dinst\ 
that includes some diagonal of $G$.
\end{definition}

For a graph $G=(V,E)$, let 
$$
\La_n(v)= \La_{G,n}(v) :=\{w \in V: d_G(v,w)\le n\}, \q
\pd \La_n(v) := \La_n(v) \setminus \La_{n-1}(v),
$$
and, furthermore, $\La_n=\La_{G,n} := \La_n(v_0)$.

\subsection{Percolation}\label{ssec:perc}

 Let $G=(V,E)$ be a connected, locally finite graph with bounded vertex-degrees. 
 A \emph{site percolation} configuration on $G$ is an 
 assignment $\omega\in \Om:=\{0,1\}^{V}$ to each vertex of either state $0$ or state $1$. 
 A vertex is called \emph{open} if it has state $1$, and \emph{closed} otherwise. 
 An \emph{open cluster} in $\omega$ is a maximal connected set of open vertices. 
 
Let $p\in [0,1]$. We endow $\Om$ with the product measure $\PP_p$ with density $p$.
For $v\in V$, let $\theta_v(p)$ be the probability that $v$ lies in an infinite open cluster.
It is standard that there exists $\pc(G)\in(0,1]$ such that
$$
\text{for } v\in V, \qq \theta_v(p) \begin{cases} = 0 &\text{if } p<\pc(G),\\
>0 &\text{if } p > \pc(G),
\end{cases}
$$ 
and $\pc(G)$ is called the \emph{critical probability} of $G$.

For background and notation concerning percolation theory, the reader is referred to the book \cite{GP}, the article
\cite{GL21a}, and to Section \ref{sec:6}. 

\section{Two criteria for property $\Pi$}\label{sec:examples}


In this section we present the \lq metric criterion' for a one-ended graph
$G\in\sQ$ to have the property $\Pi$ of Definition \ref{def10}.
This criterion is valid for one-ended,
non-triangulations $G\in\sT$, and thus we arrive in particular at the following.

\begin{theorem}\label{thm:trans}
Let $G\in\sT$ be one-ended but not a triangulation. Then $G$
has property $\Pi$.
\end{theorem}

The criterion holds for a certain class of quasi-transitive graphs,
and the outcome is a sufficient but not necessary condition for a 
quasi-transitive graph $G\in\sQ$ to have property $\Pi$, namely Theorems \ref{t72qq}. 


The embedding results of Section \ref{sec:not} may be used in proofs of the existence of
\dinst s in one-ended graphs $G\in\sQ$ satisfying the 
following forthcoming  metric criterion.
First, recall the relevant embedding property.
By Theorem \ref{p21}(a,\,c), every quasi-transitive, one-ended $G\in\sQ$ 
has a canonical embedding
in $\sH$.

Throughout  this section we shall work with the Poincar\'e disk model of hyperbolic geometry
(also denoted $\sH$), and we denote by $\rho$ the corresponding hyperbolic metric.
For definiteness, we consider only graphs $G$ embedded in the hyperbolic plane;
the Euclidean case is similar, subject to the simplification that the geometry of the space is Euclidean rather
than hyperbolic.

Let $G\in\sQ$ be one-ended and not a triangulation. By $2$-connectedness and Remark \ref{rem:emb}(d),
the faces of $G$ are bounded by cycles. 
As before, we restrict ourselves to the case when $G$ is non-amenable, and
we embed $G$ canonically in the Poincar\'e disk $\sH$. The edges of $G$ are hyperbolic geodesics, 
but its diagonals are not generally so. The  hyperbolic length of
an edge $e\in E_*\sm E$ does not generally equal the hyperbolic distance between its endvertices,
denoted $\rho(e)$.

For $e\in E_*$, let $\Gamma_e$ denote the doubly-infinite hyperbolic geodesic 
of $\sH$ passing though the endvertices of $e$,
and denote by $\pi(x) = \pi_e(x)$ the orthogonal projection of $x\in \sH$ onto $\Ga_e$.

\begin{definition}\label{def:max}
An edge $e\in E_*$ is called \emph{maximal} if 
\begin{equation}\label{eq:106}
\rho(e) \ge \rho(\pi_e(x),\pi_e(y)), \q\text{for all } f=\langle x,y\rangle\in E.
\end{equation}
The graph $G$ is said to satisfy the \emph{metric criterion} if $G$ has a canonical embedding in $\sH$ for which some diagonal $d\in E_*\sm E$ is maximal. 
\end{definition}

\begin{figure}
\centerline{\includegraphics[width=0.3\textwidth]{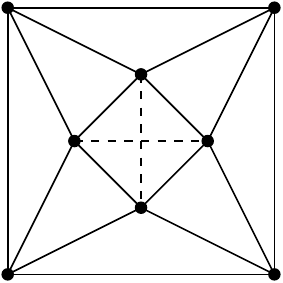}}
\caption{The graph $G$ is the tiling of the plane with copies of this square.
Taking into account the symmetries of the square, 
this tiling is canonical after a suitable rescaling of the interior square.
The diagonals are indicated by dashed lines.}\label{fig:notnec}
\end{figure}

There always exists some maximal edge of $E_*$, but it is not generally unique, and it may not be a diagonal. 
The following lemma is proved in the same manner as the forthcoming Lemma \ref{lem:22}.

\begin{lemma}\label{lem:75}
Let $e\in\argmax\{\rho(f): f\in E_*\}$. The edge $e$ is maximal.
\end{lemma}

Here is the main theorem for quasi-transitive graphs using the metric method.

\begin{theorem}\label{t72qq}
Let $G\in\sQ$ be one-ended but not a triangulation. 
Assume that $G$ satisfies the metric criterion of Definition \ref{def:max}. 
Then $G$ has the property $\Pi$ of Definition \ref{def10}.
\end{theorem}

See Sections \ref{ssec:proof72} and \ref{ssec:qtemb}
for the proofs of Theorem \ref{thm:trans}, Lemma \ref{lem:75},  and 
Theorem \ref{t72qq} by the metric method.

\begin{remark}\label{rem:nonnece}
The condition of Theorem \ref{t72qq} is sufficient but not necessary, as indicated
by the following example. Let $G$ be the canonical tiling of $\RR^2$ illustrated in Figure \ref{fig:notnec}. 
By inspection,
no diagonal is maximal, whereas $G$ has property $\Pi$.
The sufficient condition in question can be weakened as explained in Remark \ref{rem:nonnece2},
and the above example satisfes the weaker condition.
\end{remark}

\section{Some observations}\label{sec:5}

\subsection{Oxbow-removal}\label{ssec:4.1}
We begin by describing a technique of loop-removal  (henceforth referred to as \lq oxbow-removal').
Let $H$ be a simple graph embedded in the Euclidean/hyperbolic plane $\sH$ (possibly with crossings). 

\begin{lemma}\label{lem:cred}
Let $H$ be a graph embedded in $\sH$.
\begin{letlist}
\item
Let $C$ be a plane cycle of $H$ that surrounds a point $x\notin H$. 
There exists a non-empty subset 
$C'$ of the vertex-set of $C$ that forms a plane, \nst\ cycle of $H$ and surrounds $x$.
\item
Let $\pi$ be a finite (\resp, infinite) path with endpoint $v$.
There exists a non-empty subset 
$\pi'$ of the vertex-set of $\pi$ that forms a finite (\resp, infinite)  \nst\ path of $H$ starting at $v$.
If $\pi$ is finite, 
then $\pi'$ can be chosen with the same endpoints as $\pi$.
\end{letlist}
\end{lemma}

\begin{proof}
(a) Let $C=(v_0,v_1,\dots,v_n, v_{n+1}=v_0)$ be a plane cycle of $H$ that surrounds $x\notin H$; we shall apply an 
iterative process of \lq loop-removal' to $C$, and may assume $n\ge 4$. 
We start at $v_0$ and move around $C$ in 
increasing order of vertex-index. 
Let $J$ be the least $j \le n$ such that there exists $i\in\{1,2,\dots,j-2\}$ 
with $v_i\sim v_J$, and let $I$ be the earliest such $i$. 
Consider the two cycles 
$C'=(v_I,v_{I+1},\dots,v_{J},v_I)$ and $C''= (v_J,v_{J+1},\dots, v_0,v_1.\dots v_I,v_J)$.
(These cycles are called \emph{oxbows} since they arise through cutting 
across a bottleneck of the original cycle $C$.)
Since $C$ surrounds $x$, so does at least one of $C'$ and $C''$, and we suppose for 
concreteness that $C''$ surrounds $x$. 
We replace $C$ by $C''$.  This process is iterated until no such oxbows remain. 

(b) This part is proved by a similar argument. When the endpoints $v_0$, $v_n$ of $\pi$ are not neighbours,
we use oxbow-removal as above; otherwise, we set $\pi'= (v_0, v_n)$. 
\end{proof}

\begin{remark}\label{rem:nst}
Lemma \ref{lem:cred} will be used in the following context. 
Firstly, one may apply oxbow-removal to certain paths of a planar graph in order to obtain a 
non-self-touching subpath (see the forthcoming Lemma
\ref{prop:10}).  Similarly, oxbow-removal may sometimes
be used to generate a non-self-touching subpath of a concatenation of two non-self-touching paths.
\end{remark}

Path-surgery will be used in the forthcoming proofs: that is, the replacement of certain paths by others.
Consider a one-ended $G\in\sQ$, embedded canonically in the hyperbolic plane $\sH$,
which for concreteness we consider here in the Poincar\'e disk model (see \cite{CFKP}), also denoted $\sH$.
By Theorem \ref{p21}(c), every automorphism of $G$ extends to an isometry of $\sH$.
Let $\sF$ be the set of faces of $G$. For $F\in \sF$ and $x,y\in V(\pd F)$, let $\sL_{x,y}$
be the set of rectifiable curves with endpoints $x$, $y$ whose interiors are subsets of $F^\circ\sm E$,
and write
$l_{x,y}$ for the infimum of the hyperbolic lengths of all $l\in\sL_{x,y}$. 
Let 
$$
\diam(F)=\sup\{l_{x,y}: x,y \in V(\pd F)\},
$$
and 
\begin{equation}\label{eq:rho}
\Md=\max\{\diam(F): F\in\sF\}.
\end{equation}
By the properties of $G$, and in particular Theorem \ref{p21}(c), we have $\Phi<\oo$. 

Let $L$ be a geodesic of $\sH$ with endpoints in the boundary of $\sH$.
Denote by $L_\de$ the 
closed, hyperbolic $\de$-neighbourhood of $L$ (see Figure \ref{fig:longpath});
we call $L_\delta$ a \emph{hyperbolic tube}, and we say $L_\de$ 
has \emph{width} $2\de$.
Write $\pd^+ L_\de$ and $\pd^- L_\de$ for the two
boundary arcs of $L_\de$. 
An arc $\gamma$ of $\sH$ is said to \emph{cross $L_\de$ laterally} if it intersects both $\pd^+ L_\de$ and
$\pd^-L_\de$.
A path $\pi=(\dots,x_{-1},x_0,x_1,\dots)$ of $G$ (or $\Gd$) 
is said to \emph{cross $L_\de$
in the long direction} if, for any arc $\gamma$ that crosses $L_\de$ laterally and intersects no vertex of $G$, 
the number of intersections between $\gamma$ and $\pi$, if finite,  is odd.  

\begin{lemma}\label{prop:10}
Let $G=(V,E)\in\sQ$ be one-ended and embedded canonically in the Poincar\'e disk $\sH$,
and let $L_\de$ be a hyperbolic tube.
\begin{letlist}
\item If $2\de>\Md$, then $L_\de$ contains a \dinst\ of $G$, and a \dinst\ of  $G_*$,  that cross $L_\de$ in the long direction.

\item There exists $\zeta=\zeta(G)$ (depending on $G$ and its embedding) 
such that, for $r>\zeta$ and $v\in V$,  
the annulus $\La_r(v)\sm\La_{r-\zeta}(v)$
contains a
\nst\ cycle of $G$ (\resp, $G_*$) denoted $\si_r(v)$ (\resp, $\si^*_r(v)$)
 such that $v\in \si_r(v)^\circ$ (\resp, $v\in \si^*_r(v)^\circ$).
\end{letlist}
\end{lemma}

A more refined result may be found in Section \ref{sec:emb2}.

\begin{figure}
\centerline{\includegraphics[width=0.45\textwidth]{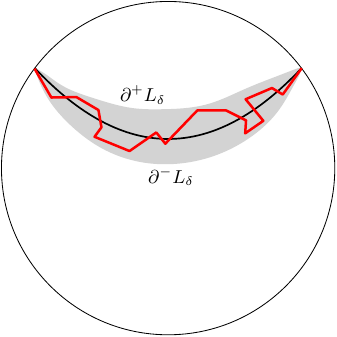}}
\caption{An illustration of Lemma \ref{prop:10}. 
The jagged (red) path crosses $L_\de$ in the long direction.}\label{fig:longpath}
\end{figure}

\begin{proof}
(a) Since all faces of $G$ are bounded, there exist vertices of $G$ in both components of $\sH\sm L_\de$. 
Now, $L_\de$ fails to be crossed in the long direction if and  only if it contains some arc 
$\gamma$ that traverses it laterally and that 
intersects no edge of $G$.
To see the \lq only if' statement, let $V^-$ and $V^+$ be the subsets of $V\cap L_\de$ that
are joined in $G\cap L_\de$ to the two boundary points of $L$, \resp; if $V^-\cap V^+=\es$, then
there exists such $\ga$ separating $V^+$ and $V^-$ in $L_\de$. 
For this $\gamma$, there exists a face $F$ and points $x,y\in V(\pd F)$, such that
$\gamma\subseteq \lambda$ for some $\lambda\in\sL_{x,y}$. 
Let $\eps\in(0,2\de-\Md)$, 
and find $\lambda'\in\sL_{x,y}$
with length not exceeding $l_{x,y}+\eps$. We may replace $\gamma$ by 
some subarc $\gamma'$ of $\lambda'\cap L_\de$.
The length of $\gamma'$ is no greater than $\Md+\eps<2\de$, a contradiction
since $L_\de$ has width $2\de$.
Therefore, $L_\de$ contains some path $\pi$ of $G$ that crosses $L_\de$ in the long direction. 

We apply oxbow-removal in $G$ to
$\pi$ as described in the proof of Lemma \ref{lem:cred}. For any arc $\gamma$ that crosses $L_\de$
laterally and intersects no vertex of $G$,
the number of intersections between $\gamma$ and $\pi$, if finite,
decreases by a non-negative, even number whenever an oxbow is removed.
It follows that the \nst\ path $\pi'$ (obtained after oxbow-removal) 
crosses $L_\de$ in the long direction.
The same conclusion applies to $G_*$ on letting $\pi$ be a path of $G_*$.

\noindent(b)
Let $\zeta$ be such that $\rho(u,v)\ge 2\Md$ whenever $d_G(u,v)\ge\zeta$.
The proof of part (b) follows that of part (a).
\end{proof}

\subsection{Graph properties}\label{ssec:4.2}

The proofs of this article make heavy use of path-surgery which, in turn, relies in part on the property of planarity.

\begin{lemma}\label{olpi}
Let $G\in\sQ$, and let $\pi$ be a (finite or infinite) \nst\ path of $G_*$. 
\begin{letlist}
\item For every face $F$ of $G$, $\pi$ contains either
zero or one or two vertices of $F$. If $\pi$ contains two such vertices $u$, $v$, then it contains also the corresponding 
edge $\langle u,v\rangle$, which may be either an edge of $G$ or a diagonal. 
\item The path $\pi$ is plane when viewed as a graph.
\end{letlist}
\end{lemma}

\begin{proof}
Let $F$ be a face. The path $\pi$ cannot contain three or more vertices of $F$, since that contradicts the
\nst\ property. Similarly, if $\pi$ contains two such vertices, it must contain also the
corresponding edge.
If $\pi$ is non-plane, it contains two or more diagonals of some face, which, by the above, cannot occur.  
\end{proof}

\begin{figure}
\centerline{\includegraphics[width=0.6\textwidth]{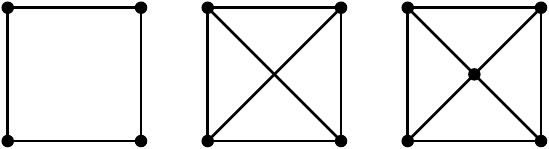}}
\caption{A square of the square lattice, its matching graph, and with its facial site added.}\label{fig:HK}
\end{figure}

As a device in the proof of Theorem \ref{maintheorem}, we shall work with the 
graph $\Gd$ obtained from $G=(V,E)$ by adding a vertex at the centre of each face $F$,
and adding an edge from every vertex in the boundary of $F$ to this central vertex.
These new vertices are
called \emph{facial sites}, or simply \emph{sites} in order
to distinguish them from the \emph{vertices} of $G$. The facial site in the face $F$ is denoted $\phi(F)$. See 
\cite[Sec.\ 2.3]{K82}, and also Figure \ref{fig:HK}. If $\langle v,w\rangle$ is a diagonal of $G_*$, it
lies in some face $F$, and we write $\phi(v,w)=\phi(F)$ for the corresponding facial site.

The main reason for working with $\Gd$ is that it serves to interpolate between $G$ and $G_*$
in the sense of \eqref{eq:three} below: 
we shall assign a parameter $s\in[0,1]$
to the facial sites in such a way that $s=0$ corresponds to $G$ and $s=1$ to $G_*$. It will also be useful that
$\Gd$ is planar whereas $G_*$ is not.

Next, we specify some desirable properties of the graphs $G_*$ and $\Gd$.
Recall the property $\Pi$ of Definition \ref{def10}.

\begin{definition}\label{def:13}
The graph  $G\in\sQ$ is said to have property
$\Pif$ if $\Gd$ has a \dinst\  including some facial site.
\end{definition}

\begin{lemma}\label{thm:rels}
Let $G\in\sQ$ be one-ended. Then
$\Pi\Rightarrow \Pif$.
\end{lemma}

\begin{proof}
Let $G$ have property $\Pi$ and let $\pi$ be a \dinst\ of $G_*$. For any two consecutive vertices $u$, $v$ of $\Pi$ 
such that $\de(u,v)$ is a diagonal, we add between $u$ and $v$ the facial site $\phi(u,v)$. The result is a 
doubly-infinite path $\pi'$ of $\Gd$. By Lemma \ref{olpi}, $\nu'$ is \nst\ in $\Gd$, whence $G$ has property $\Pif$.
\end{proof}

The properties of Definition \ref{def:13} 
are \lq global' in that they concern the existence of \emph{infinite} paths.
It is sometimes preferable to  work in the proofs with \emph{finite} paths, and to that end we introduce corresponding \lq local'  properties. 

Let $\zeta(G)$ be as in Lemma \ref{prop:10}(b). We shall
make reference to the \nst\ cycles $\si_r(v)$, $\si_r^*(v)$ given in that lemma.
We write $\what \si_r(v)$ for the \nst\ cycle of $\Gd$ obtained from $\si_r^*(v)$ 
by replacing any diagonal by a path of length $2$ passing via the appropriate facial site of $\Gd$. 
We abbreviate the closure of the region surrounded 
by $\si_r^*$ (\resp, $\what\si_r$) to $\ol\si_r^*$ (\resp, 
$\ol{\what\si}_r$).
Let $A(G)$ be the real number given as
\begin{equation}\label{eq:ag}
A(G) = \zeta(G)+\max\{d_G(u,w): \langle u,w\rangle \in E_*\sm E\}.
\end{equation}

\begin{definition}\label{def10'}
Let $A\in\ZZ$, $A> A(G)$, and let $G\in\sQ$ be one-ended.
\begin{letlist}
\item The graph $G$ is said to have property $\Pi_A$ if there exists a vertex $v\in V$ 
and  a \nst\ path $\pi=(x_0,x_1,\dots,x_n)$ of $G_*$ such that 
\begin{romlist}
\item every vertex of $\pi$ lies in $\ol\si_A^*(v)$, and $x_0,x_n\in \sias(v)$,
\item there exists $i$ such that $x_i=v$,
\item the pair $v$, $x_{i+1}$ forms a diagonal of $G_*$, which is to say that $\phi:=\phi(v,x_{i+1})$
is a facial site of $\Gd$.
\end{romlist}
\item 
The graph $G$ is said to have property $\what\Pi_{A}$ if there exist vertices $v,w\in V$ 
and  a \nst\ path $\pi=(x_0,x_1,\dots,x_n)$ of $\Gd$ such that 
\begin{romlist}
\item every vertex of $\pi$ lies in $\ol{\what\si}_A(v)$, and $x_0,x_n\in \what\si_A(v)$,
\item there exists $i$ such that $x_i=v$, $x_{i+2}=w$,
\item $x_{i+1}$ is the facial site $\phi(v,w)$ of $\Gd$.
\end{romlist}
\end{letlist} 
\end{definition}

\begin{figure}
\centerline{\includegraphics[width=0.35\textwidth]{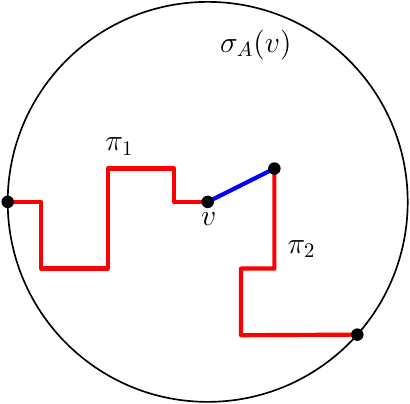}}
\caption{An illustration of the property $\Pi_A$:
a \nst\ path of $G_*$ containing a diagonal near its middle.}\label{fig:def10}
\end{figure}

That is to say, $G$ has property $\Pi_A$ (\resp, $\what\Pi_A$)
if $G_*$ (\resp, $\Gd$) contains a finite, \nst\ path of sufficient length that contains
some diagonal (\resp, facial site). This definition is illustrated in Figure \ref{fig:def10}.
Note that $\Pi_{A+1}$ (\resp, $\Pif_{A+1}$) implies $\Pi_A$ (\resp, $\Pif_A$) for sufficiently large $A$.

\begin{theorem}\label{thm:13}
Let $G\in\sQ$ be one-ended.
There exists $A'(G)\ge A(G)$ such that, for $B > A'(G)$, we have 
$\Pi \Leftrightarrow \Pi_B$ and $\Pi \Rightarrow \what\Pi_B$.
\end{theorem}

The proof of this useful theorem utilises some methods of path-surgery that will be important later, and it
is given next.

\subsection{Proof of Theorem \ref{thm:13}}\label{ssec:4.3}

(a) Let $A>A(G)$.
First, we prove that $\Pi\Leftrightarrow \Pi_A$.
Evidently, $\Pi\Rightarrow\Pi_A$.
Assume, conversely, that $\Pi_A$ holds for some $A > A(G)$. 
Let the \nst\ path $\pi=(x_0,x_1,\dots,x_n)$ of $G_*$, the vertex $v=x_i$,
and the diagonal $d=\langle v,x_{i+1}\rangle$ be as in Definition \ref{def10'}(a); 
think of $\pi$ as a directed
path from $x_0$ to $x_n$, and note by Lemma \ref{olpi} that $\pi$ is a plane graph. 
We abbreviate $\sias(v)$ to $\sias$. Let
$$
\pd^-\sias = \{y\in(\sias)^\circ: \dgs(y,\sias)=1\}.
$$

\begin{figure}
\centerline{\includegraphics[width=0.8\textwidth]{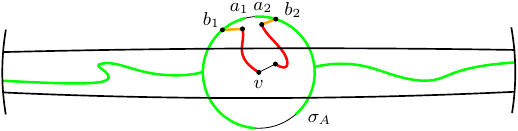}}
\caption{In the easiest case when $D \ge 2$, one finds (green) \nt\ subarcs $\si_A^i$ of $\si_A$ to which 
$v$ may be connected by \nst\ paths. These subarcs may be 
connected to the boundary of $\sH$ using subpaths of a doubly-infinite path constructed using Lemma \ref{prop:10}(a).}\label{fig:longpath2}
\end{figure}

Let $\pi_1$ be the subpath of $\pi$ from $v$ to $x_0$, and $\pi_2$ that from $x_{i+1}$ to $x_n$. 
Let $a_i$ be the earliest vertex/site of $\pi_i$ lying in $\pd^-\si_A$. See the
central circle of Figure \ref{fig:longpath2}.
We claim the following.
\begin{equation} \label{eq:text2'} 
\text{\begin{tabular}{p{0.85\linewidth}}
There exist two \nt\ subpaths $\si^1$, $\si^2$ of $\sias$, each of length at least
$\tfrac12|\sias|-4$, such that: (i) for $i=1,2$,  the subpath of $\pi_i$ leading to $a_i$ may be extended beyond
$a_i$ along $\si^i$
to form a \nst\ path ending at any prescribed $y_i\in\si^i$, and (ii) the composite path thus created (after oxbow-removal if necessary) is \nst.
\end{tabular}
}
\end{equation}

The proof of \eqref{eq:text2'} follows. Let 
\begin{equation}\label{eq:51}
A_i=\{b\in \sias: \dgs(a_i,b)=1\}, \q 
D=\max\{\dgs(b_1,b_2): b_1\in A_1,\, b_2\in A_2\}.
\end{equation}

\smallskip\noindent
{\bf Suppose $D\ge2$}. Choose $b_i\in A_i$ such that $\dgs(b_1,b_2)\ge 2$.  
As illustrated in the centre of 
Figure \ref{fig:longpath2}, we may find a \nt\ pair of \nst\ subpaths of $\sias$
such that the conclusion of \eqref{eq:text2'} holds. Some oxbow-removal may be needed at the junctions of paths
(see Remark \ref{rem:nst}).

\begin{figure}
\centerline{\includegraphics[width=0.4\textwidth]{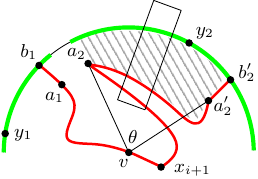}}
\caption{An illustration of the case $D=1$.
The green lines indicate the subpaths $\si_A^i$.
The rectangle is added  in illustration of the case $\theta\ge\frac34\pi$.}\label{fig:hard2'}
\end{figure}

\smallskip\noindent
{\bf Suppose $D=1$}. We may picture $\sias$ as a (topological) circle with centre $v$, and for
concreteness we assume that $a_2$ lies clockwise of $a_1$ around $\si^*_A$ (a similar argument holds if not).
See Figure \ref{fig:hard2'}.  
\begin{Alist}
\item Suppose the path  $\pi_1$, when 
continued beyond $a_1$, passes at the next step to some $b_1\in A_1$, 
and add $b_1$ to obtain a path denoted $\pi_1'$.   

Since $D=1$, the next step of $\pi_2$ beyond $a_2$ is not into $A_2$.
On following $\pi_2$ further, it moves inside $(\sias)^\circ$ until it arrives 
at some point $a_2'\in\pd^-\sias$ having some neighbour $b_2'\in\sias$ satisfying 
$\dgs(b_1,b_2')\ge 2$;
we then include the subpath of $\pi_2$ between $a_2$ and $b_2'$ to obtain a path denoted $\pi_2'$.

We declare $\si^1$
to be the subpath of $\sias$ starting at $b_1$ and extending a total distance $\frac12|\sias|-4$ around $\sias$ anticlockwise.
We declare $\si^2$ similarly to start at distance $2$ clockwise of $b_1$ and to have the same
length as $\si^1$.  

Let $\theta\in(0,2\pi)$ be the angle subtended by the
vector $\overrightarrow{a_2a_2'}$ at the centre $v$.
If $\theta<\frac34\pi$, say, each $\pi_i'$ may be 
extended along $\si^i$ to end at any prescribed $y_i\in\si^i$. Therefore, claim 
\eqref{eq:text2'} holds in this case. 

The situation can be more delicate if $\theta\ge \frac34\pi$, since then $a_2'$ may be near to $\si^1$.
By the planarity of $\pi$, the region $R$ between $\pi_2'$ and $\sias$ contains no point of $\pi_1'$ ($R$ is the shaded region in Figure \ref{fig:hard2'}).
We position a hyperbolic tube of width greater than $\Md$ 
in such a way
that it is crossed laterally by both $\pi_2'$
and the path $\si^2$ (as illustrated in Figure \ref{fig:hard2'}). 
By   Lemma \ref{prop:10}(a), this tube is crossed in the long direction by some 
path $\tau$ of $G$. The union of $\pi_2'$ and $\tau$ contains a \nst\ path $\pi_2''$ 
of $G_*$ from $x_{i+1}$ to $\si^2$
(whose unique vertex in $\si^2$ is its second endpoint). Claim \eqref{eq:text2'} follows in this situation.

\item
Suppose the hypothesis of part A does not hold, but instead $\pi_2$ passes
from $a_2$ directly into $\si^*_A$. In this case we follow A above with $\pi_1$ and $\pi_2$
interchanged.   

\item
Suppose neither $\pi_i$ passes from $a_i$ in one step into $\sias$. We add $b_2$ to the subpath from $x_{i+1}$
to $a_2$,  and continue as in part A above.

\end{Alist}

\smallskip\noindent
{\bf Suppose $D=0$}. Statement \eqref{eq:text2'} holds by a similar argument to that above.

\smallskip
Having located the $\si^i$ of \eqref{eq:text2'}, we position a hyperbolic tube as in Figure \ref{fig:longpath2},
to deduce (after oxbow-removal, see Remark \ref{rem:nst})  the existence of a \dinst\ of $G_*$ that contains the diagonal $d$. Therefore, $G$ has
property $\Pi$, as required. 

Hyperbolic tubes are superimposed on the graph at two steps of the argument above, 
and it is for this reason that we need $A$ to be sufficiently large, say $A>A'(G)$.

(b) It remains to show that $\Pi\Rightarrow \what\Pi_{A}$ for large $A$.
By Lemma \ref{thm:rels},
$\Pi\Rightarrow \Pif$, and it is immediate that
$\Pif\Rightarrow \what\Pi_{A}$ for large $A$.

\section{Proof of Theorem \ref{maintheorem}}\label{sec:6}

Consider site percolation on $G$ with product measure $\PP_p$, and fix some
vertex $v_0$ of $G$. We write $v\lra w$ if
there exists a path of $G$ from $v$ to $w$ using only open sites (such a path is called \emph{open}), and $v\lra\infty$ if there exists an infinite, open path starting at $v$.
The \emph{percolation probability} is the function $\theta$ given by
\begin{equation}\label{new34}
\theta(p)=\theta(p;G) = \PP_p(v_0 \lra \oo),
\end{equation}
so that the (site) critical probability of $G$ is $\pc(G):= \sup\{p: \theta(p)=0\}$.
The quantities $\theta(p;G_*)$ and $\pc(G_*)$ are defined similarly.

\begin{remark}\label{rem:4}
It is an old problem dating back to \cite{bs96} to decide which graphs $G$ satisfy $\pc(G)<1$,
and there has been a series of related results since.
It was proved in \cite[Thm 1.3]{DGRSY} that $\pc(G)<1$ for all
quasi-transitive graphs $G$ with super-linear growth (see also \cite{EH}). This class includes all $G \in\sQ$
with either one or infinitely many ends (see \cite[Sect.\ 1.4]{Bab97} and Theorem \ref{p21}).
\end{remark}

\begin{theorem}\label{thm1}
Let $G \in\sQ$ be one-ended. 
\begin{letlist}
\item Let $A_0\in\ZZ$. If, for every $A>A_0$, $G$ does not have property $\Pi_A$,  
then $\pc(G_*)=\pc(G)$.
\item 
There exists $A'(G)\ge A(G)$ such that the following holds.
Let $A>A'(G)$. If $G$ has property $\what\Pi_{A}$, then $\pc(\Gd) < \pc(G)$.
\end{letlist}
\end{theorem}

The constant $A'(G)$ in part (b) depends on the \emph{embedded graph} $G$, viewed as a subset of $\sH$, rather 
on the graph $G$ alone.
In advance of giving the proof of Theorem \ref{thm1}, we explain how
it implies Theorem \ref{maintheorem}.

\begin{proof}[Proof of Theorem \ref{maintheorem} (assuming Theorem 5.2)]
If $G$ does not have property $\Pi$, by Theorem \ref{thm:13} for large $A$  it does not have property $\Pi_A$,
whence by Theorem \ref{thm1}(a), $\pc(G_*)=\pc(G)$. Conversely, if $G$ has property $\Pi$,
by Theorem \ref{thm:13} again it has property $\what\Pi_{A}$ for large $A$, whence by  
Theorem \ref{thm1}(b),  $\pc(\Gd) < \pc(G)$. The final claim follows by the elementary inequality
$\pc(G_*) \le \pc(\Gd)$; see \eqref{eq:three}.
\end{proof}

\begin{proof}[Proof of Theorem \ref{thm1}(a)]
Let $A_0\in\ZZ$. Assume $G$ has property $\Pi_A$ for no $A\ge A_0$, and let $p>\pc(G_*)$. 
Let $\pi$ be an infinite open path of $G_*$ with some endpoint  $x$.  By Lemma \ref{lem:cred}(b),
there exists a subset $\pi'$ of $\pi$ that forms a \nst\ path of $G_*$ with endpoint $x$. Let $A>A_0$. 
Since $\Pi_A$ does not hold, every edge of $\pi'$ at distance $2A$ or more from $x$ is an edge of $G$, so that
there exists an infinite open path in $G$. Therefore, $p\ge \pc(G)$, whence
$\pc(G_*)=\pc(G)$. 
\end{proof}

The rest of this section is devoted to the proof of Theorem \ref{thm1}(b).
Let 
$\what\Om = \Om_V \times \Om_\Phi$ where $\Phi$
is the set of facial sites and $\Om_\Phi=\{0,1\}^\Phi$. 
For $\what\om=\om\times \om' \in \what \Om$ and $\phi\in \Phi$, we call $\phi$ \emph{open}
if $\om'_{\phi}=1$, and \emph{closed} otherwise. Let $\PPps=\PP_p \times \PP_s$
be the corresponding product measure on $\Om_V\times\Om_\Phi$, and
$$
\theta(p,s) = \lim_{n\to\oo} \theta_n(p,s)\qq\text{where}\qq 
\theta_n(p,s)=\PPps(v_0 \lra \pd\La_{n}\text{ in } \Gd),
$$
so that
\begin{equation}\label{eq:three}
\theta(p,0)=\theta(p;G), \q \theta(p,p)=\theta(p;\Gd),
\q \theta(p,1)=\theta(p;G_*),
\end{equation}
where $\theta(p;H)$ denotes the percolation probability
of the graph $H$.
Note that $\theta(p,s)$ is non-decreasing in $p$ and $s$.
The following proposition implies Theorem \ref{thm1}(b).

\begin{proposition}\label{thm2}
There exists $A'(G)<\oo$ such that the following holds.
Suppose $G \in \sQ$ is one-ended and has property $\what\Pi_{A}$
where $A>A'(G)$. 
Let $s\in(0,1)$. There exists $\eps=\eps(s)>0$ such that
$\theta(p,s)>0$ for $\pc(G)-\eps<p<\pc(G)$.
\end{proposition}

We do not investigate the details of how $A'(G)$ depends on $G$. An explicit lower bound on $A'(G)$
may be obtained in terms of local properties of the embedding of $G$, but it is doubtful whether this will be useful in practice.

The rest of this proof is devoted to an outline of that of Proposition \ref{thm2}.
Full details are not included, since they are very close to established arguments of \cite{AG}, \cite[Sect.\ 3.3]{GP},  and elsewhere.

Let $n$ be large, and later we shall let $n \to\oo$. 
Consider site percolation on $\Gd$  with measure $\PPps$. 
We call a vertex (\resp, facial site) $z$ \emph{pivotal} if it is pivotal for the existence of
an open path of $\Gd$ from $v_0$ to $\pd\La_n$ (which is to say that such a path exists if 
$z$ is open, and not otherwise).
Let $\piv_n$ be the set of pivotal vertices, and $\div_n$ the set of pivotal facial sites.
Proposition \ref{thm2} follows in the \lq usual way' (see \cite[Sect.\ 3.3]{GP})
from the following statement.

\begin{lemma}\label{lem:1}
Let $p,s\in(0,1)$. There exists $M\ge 1$ and $f:(0,1)^2\to(0,\oo)$ such that, for $n>4M$ and 
every $z\in \La_n$,
\begin{equation}\label{3}
\PPps(z \in \piv_n) \le f(p,s) \PPps(\div_n \cap \La_M(z) \ne \es).
\end{equation}
\end{lemma}

\begin{proof}[Proof of Proposition \ref{thm2} (assuming Lemma \ref{lem:1})]
On summing \eqref{3}
over $z \in \La_n$, we obtain by Russo's formula (see \cite[Sec.\ 2.4]{GP}) 
that there exists $g(p,s)<\oo$ such that
\begin{equation}\label{39}
\frac{\pd}{\pd p}\theta_n(p,s) \le  g(p,s) \frac{\pd}{\pd s}\theta_n(p,s).
\end{equation}

The derivation of Proposition
\ref{thm2} from the above differential inequality  is explained in \cite{AG} and \cite[p.\ 60]{GP}. 
It proceeds as follows. 
Let $\eta>0$ be small, and find $\gamma\in(0,\oo)$ such that 
$g(p,s)\le 1/\gamma$ on $[\eta ,1-\eta ]^2$. Let $\psi\in [0,\frac12\pi)$
satisfy $\tan\psi=\g^{-1}$.

At the point $(p,s)\in [\eta,1-\eta ]^2$, 
the rate of change of $\theta_n(p,s)$ in the direction $(\cos\psi 
,-\sin\psi )$ satisfies
\begin{align} 
\nabla\theta_n \cdot(\cos\psi ,-\sin\psi )&=\frac{\pd\theta_n}{\pd p}\cos\psi
-\frac{\pd\theta_n}{\pd s}\sin\psi  
\label{tcn.8}\\ 
&\leq\frac{\pd\theta_n}{\pd p} (\cos\psi -\gamma\sin\psi )= 0 
\nonumber
\end{align}
by \eqref{39}, since $\tan\psi=\gamma^{-1}$. 

Suppose now that $(a,b)\in[2\eta,1-2\eta ]^2$. Let
$$
(a',b') = (a,b) - \eta (\cos\psi ,-\sin\psi ),
$$
noting that
$(a',b')\in [\eta ,1-\eta ]^2$. 
By integrating \eqref{tcn.8} along the line segment joining
$(a',b')$ to $(a,b)$, we obtain that
\begin{equation*}
\theta (a',b')=\lim_{n\to\infty}\theta_n(a',b')\geq\lim_{n\to\infty}\theta_n(a,b)
=\theta(a,b). 
\end{equation*}
Now let $s\in [2\eta,1-2\eta]$ and let $\eps\in(0,s)$ be small.
Take  $(a,b)=(\pc +\eps,s-\eps)$ where $\pc=\pc(G)$,
and define $(a',b')$ as above.
We choose $\eps$ sufficiently small 
that $(a,b), (a',b')\in [2\eta,1-2\eta ]^2$,
and that $a'<\pc$. 
The above calculation yields that
\begin{equation*}
\theta (a',b')\ge \theta(\pc+\eps,s-\eps) \ge \theta(\pc+\eps,0)>0, 
\end{equation*}
as required.
\end{proof}

Here is an outline of the proof of Lemma \ref{lem:1}
(a more formal proof follows this outline). 
Let $\what\om\in\what\Om$, $z\in V\cap\La_n$, and suppose 
\begin{equation}\label{piv}
\text{$z$ is open and pivotal in the configuration $\what \om$.}
\end{equation}

 By making changes to the configuration $\what\om$ 
within the box $\La_{4M}(z)$ for some fixed $M$, 
\begin{equation} \label{eq:text0} 
\text{\begin{tabular}{p{0.85\linewidth}}
we  construct a configuration
in which $\La_{M}(z)$ contains a pivotal facial site.
\end{tabular}
}
\end{equation}   
This implies \eqref{3} with $f$ depending on the choice of $z$.
Since $\La_{4M}(z)$ is finite and there are only finitely many types of vertex (by quasi-transitivity), 
$f$ may be chosen to be independent of $z$.
The above is achieved in five stages.

\emph{Assume for now that $\what\om\in\what\Om$ and the pivotal vertex $z$ satisfies}
\begin{equation}\label{eq:notedge}
z\in \La_{n-2M}\sm \La_{2M}.
\end{equation}
For clarity of exposition, our illustrations are drawn as if $G$ is embedded 
properly in the Euclidean rather
than the hyperbolic plane. The principal effect of this is that hyperbolic tubes are represented as Euclidean rectangles.

Let $G$ have property $\what\Pi_{A}$.
Let $\pi=(x_j)$, $v=x_i$, be as in Definition \ref{def10'}(b), 
and write $\phi=x_{i+1} = \phi(v,x_{i+2})$. Find $\alpha\in\Aut(G)$ such that $v' = \alpha v$ satisfies $d_G(z,v')\le \De$,
where $\De$ is given in \eqref{eq:defde}. Let $M=2(A+\De)$,
so that $\La_A(v')\subseteq \La_{M/2}(z)$. The outline of the proof is as follows.

\begin{romlist}
\item [I.] If there exist one or more open facial sites in $\La_M(z)$, we declare them one-by-one to be closed. If
at some point in this process, some facial site is found to be pivotal, then we have achieved \eqref{eq:text0},
by changing $\what\om$ within a bounded region.
We may therefore assume that this never occurs, or equivalently that 
\begin{equation}\label{eq:34}
\text{$\what\om$ has no open facial site in 
$\La_M(z)$.}
\end{equation}

\item [II.] Find a \nst\ open path $\nu$ in $\what\om$ from $v_0$ to $\pd\La_n$. This path passes necessarily through 
the pivotal vertex $z$.

\item [III.] By making changes within $\La_{2M}(z)$, we construct \nt\ subpaths of $\nu$ from $v_0$
(\resp, $\pd\La_n$) to $\pd\La_{M}(z)$, that can be extended inside $\La_M(z)$
in a manner to be specified at Stage V. This, and especially the following, 
stage resembles closely part of the proof in Section \ref{ssec:4.3}.

\item [IV.] We splice a copy (denoted $\pi'=\alpha\pi$) of $\pi$ inside $\La_A(v')$, and we make local changes
to obtain paths $\pi_1$, $\pi_2$ from the two endpoints
of $\alpha \phi$, \resp,  to $\pd\La_A(v')$ that can be extended 
outside $\La_A(v')$ in a manner to be specified at Stage V.

\item [V.] Between the contours $\pd\La_A(v')$ and $\pd\La_M(z)$,
we arrange the configuration in such a way that the retained parts of $\nu$ hook up with the endpoints of 
the $\pi_i$.
In the resulting configuration, the facial site $\phi':=\alpha \phi$ is pivotal.
\end{romlist}

Some work is needed to ensure that $\phi'$ can be made pivotal in the final configuration. 
Lemma \ref{prop:10}(b)
will be used to traverse the annulus between the two contours at Stage V. In making connections at junctions
of paths, we shall make use of the planarity of $\Gd$.
Rather than working with the boundaries of $\La_M(z)$ and $\La_A(v')$, we shall work instead with the 
\nst\ cycles $\wsims:=\wsims(z)$ and $\wsias:= \wsias(v')$ of $\Gd$ given in Lemma
\ref{prop:10}(b). Let
\begin{align*}
\pd^+\wsims &= \{y\in\sH\sm \ol{\what\si}_M  : \dw(y,\wsims)=1\},\\
\pd^-\wsias &= \{y\in(\wsias)^\circ: \dw(y,\wsias)=1\}.
\end{align*} 

\begin{proof}[Proof of Lemma \ref{lem:1}]
Stage I is first followed as stated above.

\begin{figure}
\centerline{\includegraphics[width=0.7\textwidth]{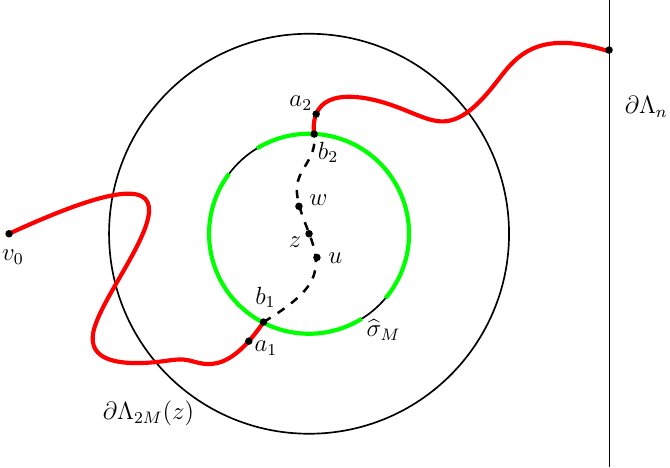}}
\caption{An illustration of the construction at Stages II/III. The \nst\ path $\nu$ contains subpaths from $v_0$  to
$\what\si_M$, and from the latter set to $\pd\La_n$. The subpaths $\si_M^i$ 
of $\what\si_M$ are indicated
in green.}\label{fig:hard1}
\end{figure}

\smallskip\noindent
{\bf Stage II.}
By \eqref{piv}, we may find an open, \nst\ path $\nu$ of $\Gd$ from $v_0$ to $\pd\La_n$, and we consider $\nu$
as thus directed. By \eqref{eq:34}, $\nu$ includes no facial site of $\La_M(z)$.
The path $\nu$ passes necessarily through $z$, and we let $u$ (\resp, $w$) be the preceding (\resp, succeeding) vertex to $z$.

For $y \in V$, and the given configuration $\what\om$ (satisfying \eqref{eq:34}), let
$$
C_y=\{x\in  V: y\lra x \text{ in $\Gd\sm\{z\}$}\},
$$
and write $C_y$ also for the corresponding induced subgraph of $\Gd$.
By \eqref{piv},
\begin{Alist}
\item $C_u$ and $C_w$ are disjoint (and also \nt),

\item the subpath of $\nu$, denoted $\nu(u-)$,  from $v_0$ to $u$ contains no facial site of $\La_M(z)$,

\item the subpath of $\nu$, denoted $\nu(w+)$,  from $w$ to $\pd\La_n$ 
contains no facial site of $\La_M(z)$,

\item the pair $\nu(z-)$, $\nu(z+)$ is \nt.
\end{Alist}

\smallskip\noindent
{\bf Stage III.}
This is closely related to the proof of Theorem \ref{thm:13} given in Section \ref{ssec:4.3}.
Note that the intersection of $\nu(u-)\cup\nu(w+)$ and $\La_{2M}(z)$ comprises a family of paths
rather than two single paths. 
See Figure \ref{fig:hard1}.

We follow $\nu(u-)$ towards $u$, and $\nu(w+)$  backwards
towards $w$, until we reach the first vertices/sites,
denoted $a_1$, $a_2$, \resp, lying in $\pd^+\wsims$. 
Let $\nu_1$ be the subpath of $\nu(u-)$ between 
$v_0$ and  $a_1$, and $\nu_2$ that of $\nu(w+)$ 
between $\pd\La_n$ and  $a_2$. 
We now change the states of certain vertices/sites $x \in \La_{2M}(z)$ by declaring 
\begin{equation}\label{nbr}
\text{every $x \in \La_{2M}(z) \sm\ol{\what\si}_M$ is declared
open if and only if $x\in\nu_1\cup \nu_2$. }
\end{equation} 

We investigate next the subsets of $\wsims$ to which the $a_i$ may be connected within $\si_M$.
We shall show that:
\begin{equation} \label{eq:text1} 
\text{\begin{tabular}{p{0.85\linewidth}}
there exist two \nt\ subpaths $\si_M^1$, $\si_M^2$ of $\wsims$, each of length at least
$\tfrac12|\wsims|-4$, such that, for $i=1,2$: (i) $a_i$ has a neighbour
$b_i \in \si_M^i$,
(ii) for $y_i\in\si_M^i$,  the path $\nu_i$ may be extended from
$b_i$ to $y_i$ along $\si_M^i$, thereby creating
(after oxbow-removal if necessary)  a \nst\ path from the other endpoint
of $\nu_i$, (iii) the composite path $\nu'_i$ thus created is \nst,
and (iv) the pair $\nu'_1$, $\nu'_2$ is \nt.
\end{tabular}
}
\end{equation}
An explanation follows. Let 
\begin{equation}\label{eq:51'}
A_i=\{b\in \wsims: \dw(a_i,b)=1\}, \q D=\max\{\dw(b_1,b_2): b_1\in A_1,\, b_2\in A_2\}.
\end{equation}

\smallskip\noindent
{\bf Suppose $D\ge2$}. Choose $b_i\in A_i$ such that $\dw(b_1,b_2)\ge 2$.  Statement \eqref{eq:text1} 
follows as illustrated in Figure \ref{fig:hard1}.

\begin{figure}
\centerline{\includegraphics[width=0.7\textwidth]{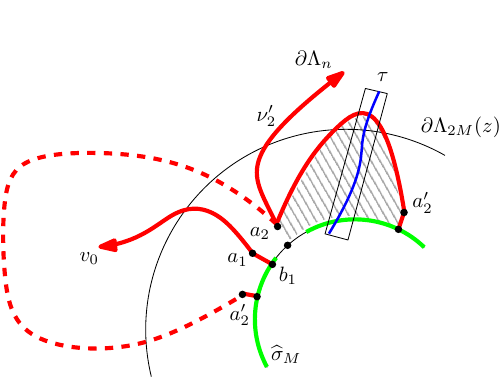}}
\caption{An illustration of the case $D=1$ in the Stage III construction.
There are two subcases, depending on whether $\theta>0$ (solid line) or $\theta<0$ (dashed line).
The green lines indicate the subpaths $\si_M^i$ in the subcase $\theta>0$.
The rectangle is added  in illustration of the hyperbolic tube used in the case $\theta\ge\frac34\pi$.}\label{fig:hard2}
\end{figure}

\smallskip\noindent
{\bf Suppose $D=1$}. We may picture $\si_M$ as a circle with centre $z$, and for
concreteness we assume that $a_2$ lies clockwise of $a_1$ around $\wsims$ (a similar argument holds if not)
See Figure \ref{fig:hard2}.  
\begin{Alist}
\item Suppose the path  $\nu_1$, when 
continued along $\nu(z-)$ beyond $a_1$, passes at the next step to some $b_1\in A_1$, 
and add $b_1$ to $\nu_1$ (to obtain a path denoted $\nu_1'$).   

Since $D=1$, the next step of $\nu(w+)$ beyond $a_2$ is not to $A_2$.
On following $\nu(w+)$ further, it moves inside $\sH\sm\ol{\what\si}_M$ until it arrives 
at some point $a_2'\in\pd^+\wsims$ having some neighbour $b_2'\in\wsims$ satisfying 
$\dw(b_1,b_2')\ge 2$; we then add to $\nu_2$ the subpath of $\nu(w+)$ between $a_2$ and $b_2'$ 
(to obtain an extended path $\nu_2'$).
Let $\theta(a_2')$ be the angle subtended by the
vector $\overrightarrow{a_2a_2'}$ at the centre $z$, 
counted positive if $\nu(w+)$ passes clockwise around $z$
of $\wsims$, and negative if 
anticlockwise.

\begin{romlist}
\item
There are two cases, depending on whether $\theta:=\theta(a_2')$
is positive or negative.  Assume first that $\theta>0$. If $\theta<\frac34\pi$, say, we declare $\si_M^1$
to be the subpath of $\wsims$ starting at $b_1$ and extending a total distance $\frac12|\wsims|-4$ around $\si_M$ anticlockwise.
We declare $\si_M^2$ similarly to start at distance $2$ clockwise of $b_1$ 
along $\what\si_M$ and to have the same
length as $\si_M^1$.  
Each $\nu_i'$ may be 
extended along $\si_M^i$ to end at any prescribed $y_i\in\si_M^i$. Therefore, claim 
\eqref{eq:text1} holds in this case. 

The situation can be more delicate if $\theta\ge \frac34\pi$, since then $a_2'$ may be near to $\si_M^1$.
By the planarity of $\nu$, the region $R$ between $\nu_2'$ and $\si_M$ contains no point of $\nu_1'$ ($R$ is the shaded region in Figure \ref{fig:hard2}).
We position a hyperbolic tube of width greater than $\Md$ 
in such a way
that it is crossed laterally by both $\nu_2'$
and the path $\si_M^2$ given above. By   Lemma \ref{prop:10}(a), this tube is crossed in the long direction by some 
path $\tau$ of $\Gd$. As illustrated in Figure \ref{fig:hard2},
the union of $\nu_2'$ and $\tau$ contains (after oxbow-removal)
a \nst\ path $\nu_2''$ from $\pd\La_n$ to $\si_M^2$
(whose unique vertex in $\si_M^2$ is its second endpoint). We now declare each vertex/site of 
$\La_{2M}(z)\sm (\what\si_M)^\circ$
to be open if and only if it lies in $\nu_1'\cup \nu_2''$. Claim \eqref{eq:text1} follows in this situation,
with the $\si_M^i$ given as above.

\begin{figure}
\centerline{\includegraphics[width=0.6\textwidth]{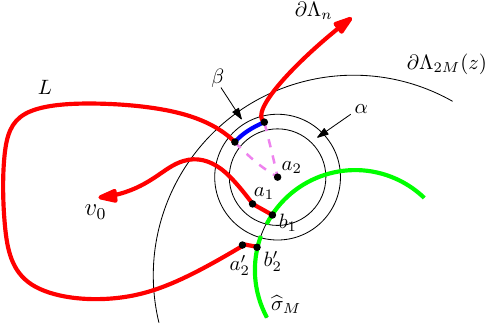}}
\caption{When $D=1$ and $\theta<0$, we adjust the path $\nu_2$ by bypassing a subpath through $a_2$.}
\label{fig:hard3}
\end{figure}

\item Assume $\theta<0$, in which case there arises a complication in the above
construction, as illustrated in Figure \ref{fig:hard3}. In this case, 
there is a sub-path $L$ of $\nu_2'$ from $a_2$ to $a_2'$, that passes anticlockwise around $v_0$,
and $\nu_1'$ contains no vertex/site outside the closed cycle comprising $L$ followed
by the subpath of $\wsims$ from $b_2'$ to $b_2$.  In order to overcome this problem, 
we alter the path $\nu_2'$ as follows. Let $\alpha$ denote the annulus $\La_M(a_2)\sm \La_{M-\zeta}(a_2)$,
with $\zeta$ as in Lemma \ref{prop:10}(b). (We may assume $M\ge 2\zeta$.)
By that lemma, $\alpha$ contains a \nst\ cycle $\beta$ of $\Gd$
that surrounds $a_2$. The union of $\nu_2'$ and $\beta$ contains (after oxbow-removal) a \nst\ path $\nu_2''$ 
of $\Gd$ from 
$\pd\La_n$ to $a_2'$ that does not contain $a_2$ (see Figure \ref{fig:hard3}). 
We declare every $x\in \nu_2''$ open and 
every $x\in\nu_2'\sm\nu_2''$ closed. The subpaths $\si_M^i$ of $\wsims$ may now be defined as above. 
\end{romlist}

\item
Suppose the hypothesis of part A does not hold, but instead $\nu_2$ passes
from $a_2$ into $\wsims$. In this case we follow A with $\nu(u-)$ and $\nu(w+)$
interchanged.  This case is slightly shorter than A since the above complication cannot occur. 

\item
Suppose neither $\nu_i$ passes from $a_i$ directly into $\wsims$. We add $b_2$ to $\nu_2$ and continue as in 
A above.

\end{Alist}

\smallskip\noindent
{\bf Suppose $D=0$}. Statement \eqref{eq:text1} holds by a similar argument to that of case (ii),

\begin{figure}
\centerline{\includegraphics[width=0.4\textwidth]{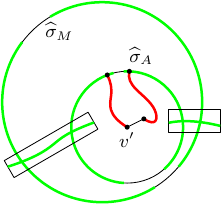}}
\caption{An illustration of the construction at Stages IV and V.}
\label{fig:hard4}
\end{figure}

\smallskip\noindent
{\bf Stage IV.}
We next pursue a similar strategy within $\La_A(v')$. The argument is essentially that in 
proof of Theorem  \ref{thm:13} given in Section \ref{ssec:4.3}, and the details of this are omitted here. 
See Figures \ref{fig:hard2'} and \ref{fig:hard4}.

\smallskip\noindent
{\bf Stage V.}
Having located the subpaths $\si_M^i$ of $\wsims$, and the subpaths $\si_A^i$ of $\wsias$,
we prove next that there exists $j\in\{1,2\}$, and \nst\ paths $\mu_1$, $\mu_2$, such that:
(i) $\mu_1$, $\mu_2$ is a \nt\ pair,   (ii) $\mu_1$ has endpoints in $\si_M^1$
and $\si_A^j$, and $\mu_2$ has endpoints in $\si_M^2$ and $\si_A^{j'}$, where $j'\in\{1,2\}$, $j'\ne j$,
and (iii) apart from their endpoints, $\mu_1$ and $\mu_2$ lie in 
$(\what\si_M)^\circ\sm \ol{\what\si}_A$. This
statement follows as in Figure \ref{fig:hard4} by positioning two hyperbolic tubes of width exceeding $\Md$,
and 
appealing to Lemma \ref{prop:10}(a). It may be necessary to remove some oxbows at the junctions of paths.

\smallskip
Hyperbolic tubes are superimposed on $\wsias$ above, and it is for this reason that $A$ is assumed 
to be sufficiently large.

Having satisfied \eqref{eq:text0} subject to \eqref{eq:notedge},
we next explain how to remove the assumption \eqref{eq:notedge}. 
Let the pivotal vertex $v$ satisfy $v\in\La_{2M}$; 
a similar argument applies if $v\in\La_n\sm\La_{n-2M}$. 
Let $\pi$ be an infinite, \nst\ open path of $\Gd$ starting at $v_0$, and declare closed every vertex of $\La_{4M}$
not lying in $\pi$. (Such a $\pi$ exists by connectivity and oxbow-removal.)
 In the resulting configuration, 
every vertex/site in the subpath of $\pi$ from $\pd\La_{2M}$ to $\pd\La_{4M}$ is pivotal. 
We pick one  such vertex 
and apply the above arguments to obtain a pivotal facial site lying in $\La_{4M}$.
\end{proof}

\section{Strict inequality using the metric method}\label{sec:emb2}

\subsection{Embeddings in the Poincar\'e disk}\label{ssec:emb}
Throughout  this section we shall work with the Poincar\'e disk model of hyperbolic geometry
(also denoted $\sH$), and we denote by $\rho$ the corresponding hyperbolic metric.

\subsection{Proof of Theorem \ref{thm:trans}}\label{ssec:proof72}

Let $\Ga$ be a doubly-infinite geodesic in the Poincar\'e disk.
Pick a fixed but arbitrary total ordering $<$ of $\Ga$. 
Then $\Ga$ may be parametrized by any function $p:\Ga\to\RR$ satisfying $p(v) = p(u)+\rho(u,v)$ for $u,v\in\Ga$, 
$u<v$, and we fix such $p$.

Any $x \notin \Ga$ has an orthogonal projection $\pi(x)$ onto $\Ga$ (for $x\in\Ga$, we set $\pi(x)=x$). 

\begin{figure}
\centerline{\includegraphics[width=0.5\textwidth]{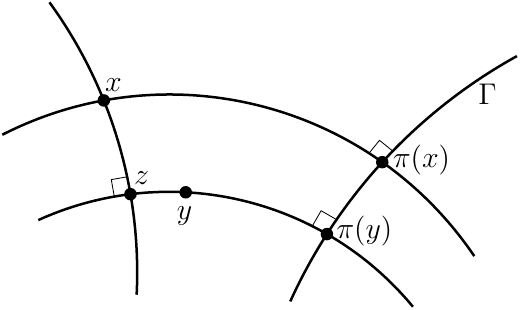}}
\caption{An illustration of the proof of Lemma \ref{lem:22}.
The four curved lines are geodesics.}\label{fig:lambert}
\end{figure}

\begin{lemma}\label{lem:22}
For $x,y\in\sH$, we have $\rho(\pi(x),\pi(y)) \le \rho(x,y)$.
\end{lemma}

\begin{proof}
We  assume for simplicity that $x$ and $y$ are distinct and lie in the same connected component of $\sH\sm\Ga$;
a similar proof holds if not. The points $x,\pi(x),\pi(y),y$ form a quadrilateral with two consecutive right angles (see Figure \ref{fig:lambert}).
Let $z$ be the orthogonal projection of $x$ onto the geodesic containing $y$ and $\pi(y)$. The triple  $x,z,y$
forms a right-angled triangle, and the quadruple $x,z,\pi(y),\pi(x)$ forms a Lambert quadrilateral. 
By the geometry of such shapes (see, for
example, \cite[Sect.\ III.5]{Iver}), we have that $\rho(x, y) \ge \rho(x,z) \ge \rho(\pi(x),\pi(y))$.
\end{proof}

Let $G=(V,E)\in\sT$ be one-ended but not a triangulation.
We shall consider only the case when $G$ is non-amenable, so that it is embedded  as an Archimedean tiling in 
the Poincar\'e disk; 
the Euclidean case is similar.
 For an edge  $e$ of $G_*=(V,E_*)$, let $\rho(e)$ denote the hyperbolic 
distance between its endvertices; since every $e$ of $G_*$ (in its embedding) is a geodesic, 
$\rho(e)$ equals the hyperbolic length of $e$.  
Since the embedding is Archimedean,  every edge of $G$ has the same hyperbolic length, and
we may therefore assume for simplicity that
\begin{equation}\label{eq:100}
\rho(e)=1, \qq e\in E.
\end{equation}
Each $e\in E_*$ is a sub-arc of a unique 
doubly-infinite geodesic, denoted $\Ga_e$, of $\sH$. 

Let $r$ be the maximal number of edges in a face of $G$, and let $F$ be a face
of size $r$. Since $F$ is a regular $r$-gon, by \eqref{eq:100}, $F$ has some diagonal $d$ satisfying
\begin{equation}\label{eq:101}
\rho(d)\ge \rho(e)\ge 1, \qq e\in E_*,
\end{equation}
and we choose $d$ accordingly. By Lemma \ref{lem:22} applied to the geodesic $\Ga_d$,
\begin{equation}\label{eq:102}
\rho(\pi(e))  \le  \rho(e)\le \rho(d), \qq e\in E_*,
\end{equation}
where $\pi$ denote orthogonal projection onto $\Ga_d$, and $\rho(\ga)$ is the hyperbolic distance
between the endpoints of an arc $\ga$.

Let $<$ and $p$ be the ordering and parametrization of $\Ga_d$ given at the start of this subsection.
We extend the domain of $p$ by setting
\begin{equation*}
p(x)= p(\pi(x)),\qq x\in\sH.
\end{equation*}
We construct next a doubly-infinite path of $G_*$ containing $d$ and lying \lq close' to $\Ga_d$.
Write $d=\langle a,b\rangle$ where $a<b$. Let $\Ga_d^+$ (\resp, $\Ga_d^-$) be the
sub-geodesic obtained by proceeding along $\Ga_d$ from $b$ in the positive direction (\resp,
from $a$ in the negative direction). As we proceed along $\Ga_d^+$, we encounter edges and faces
of $G$. If $e\in E$ is such that $e\cap \Ga_d^+\ne\es$, then the intersection is either a point
or the entire edge $e$ (this holds since both $e$ and $\Ga_d$ are geodesics).  

\begin{lemma}\label{lem:t93}
Let $e=\langle x,y\rangle\in E$ be an edge whose interior $e^\circ$ intersects $\Ga_d^+$ at a singleton $g$ only, 
so that $e^\circ\cap \Ga_d^+=\{g\}$.  Then,
\begin{letlist}
\item either $p(x)=p(g)=p(y)$, or
\item some endvertex $z\in\{x,y\}$ of $e$ satisfies $p(z)>p(g)$.
\end{letlist}
\end{lemma}

\begin{proof}
The first case arises when $e$, viewed as a geodesic, is perpendicular to $\Ga_d^+$, and the second when it is not.
See Figure \ref{fig:flat}.
\end{proof}

\begin{figure}
\centerline{\includegraphics[width=0.25\textwidth]{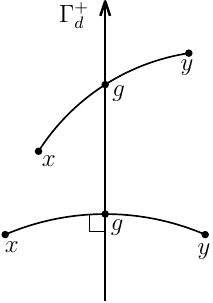}}
\caption{The two cases that arise when $\Ga_d^+$ meets an edge which is either perpendicular or not.}\label{fig:flat}
\end{figure}

In proceeding along $\Ga_d^+$, we make an ordered list $(w_i)$ of vertices as follows. 
\begin{letlist}
\item Set $w_0=b$.

\item Every time $\Ga_d$ passes into the interior of a face $F'$, it exits either at a vertex $v'$ or across
the interior of some edge $e'$. In the first case we add $v'$ to the list, and in the second, 
we add to the list an endvertex of $e'$ with maximal $p$-value. 

\item If $\Ga_d^+$ passes along an edge $e\in E$, we add both its endvertices to the list in the order
in which they are encountered.
\end{letlist}

The following lemma is proved after the end of the current proof.

\begin{lemma}\label{lem:t44}
The infinite ordered list $w=(w_0,w_1,\dots)$ is a path of $G_*$ with the property that $p(w_i)$ is strictly increasing in $i$.
\end{lemma}

We apply oxbow-removal, Lemma \ref{lem:cred}(b),  to $w$ to obtain an infinite, \nst\ path $\nu^+=(\nu_0,\nu_1,\dots)$
of $G_*$ satisfying
\begin{equation}\label{eq:103}
\nu_0=b,\qq p(\nu_0)<p(\nu_1) < \cdots. 
\end{equation} 
By the same argument applied to $\Ga_d^-$, there exists an infinite, \nst\ path
$\nu^-= (\nu_{-1}, \nu_{-2},\dots)$ of $G_*$ satisfying
\begin{equation}\label{eq:104}
\nu_{-1}=a,\qq p(\nu_{-1})>p(\nu_{-2}) > \cdots. 
\end{equation} 
The composite path $\nu$ obtained by following $\nu^-$ towards $a$, then $d$, then $\nu_+$, 
fails to be \nst\ in $G_*$ if and only if
there exists $s<0$ and $t\ge 0$ with $(s,t)\ne (-1,0)$ such that $e'':=\langle \nu_s,\nu_t\rangle \in E_*$.
If the last were to occur, by \eqref{eq:103}--\eqref{eq:104},
\begin{equation*}
\rho(\pi(e'')) = p(\nu_t)-p(\nu_s) >p(b)-p(a)=\rho(d),
\end{equation*}
in contradiction of \eqref{eq:102}. Thus $\nu$ is the required \nst\ path.
The above may be regarded as a more refined version of part of Proposition \ref{prop:10}.

\begin{proof}[Proof of Lemma \ref{lem:t44}]
That $w$ is a path of $G_*$ follows from its construction, and we turn to the second claim.
Let $m\ge 0$, and consider $w_0,w_1,\dots,w_m$ as having been identified. We claim that
\begin{equation}\label{eq:105}
p(w_m)<p(w_{m+1}).
\end{equation}

\begin{letlist}
\item Suppose $w_m\in\Ga_d^+$. 
\begin{romlist} 
\item If $\Ga_d^+$   
 includes next an entire edge of the form $\langle w_m, g\rangle\in E$, then $w_{m+1}=g$ and \eqref{eq:105} holds.

\item Suppose $\Ga_d^+$ enters next the interior of some face $F'$. If it exits $F'$ at a vertex, then this vertex
is $w_{m+1}$ and \eqref{eq:105} holds.
Suppose it exits by crossing the interior of an edge $e'$. If $w_m$ is an endvertex of $e'$, then $w_{m+1}$ is its
other endvertex  and \eqref{eq:105} holds; if not, then $w_{m+1}$ is an endvertex 
of $e'$ with maximal $p$-value 
(recall Lemma \ref{lem:t93}).
\end{romlist}
\item
Suppose $w_m$ is the endvertex of an edge $e$ that is crossed (but not traversed in its entirety) by $\Ga_d^+$,
and  let $F'$ be the face thus entered. The next vertex $w_{m+1}$ is given as in (a)(ii) above, and \eqref{eq:105} holds.
\end{letlist}
The proof is complete.
\end{proof}

Finally in this section, we prove Lemma \ref{lem:75}.

\begin{proof}[Proof of Lemma \ref{lem:75}]
Let $e=\langle u,v\rangle\in E_*$ satisfy $e\in\argmax\{\rho(f): f\in E_*\}$, 
and let $\Ga$ be the doubly infinite geodesic through $u$ and $v$.
Then, for $f=\langle x,y\rangle\in E_*$,
$$
\rho(e) \ge \rho(f) \ge \rho(x,y) \ge \rho(\pi(x),\pi(y)),
$$
where $\pi$ denotes projection onto $\Ga$. The last inequality holds by Lemma \ref{lem:22}.
Therefore, $e$ is maximal.
\end{proof}

\subsection{The case of quasi-transitive graphs}\label{ssec:qtemb}

Certain complexities arise in applying the  techniques of Section \ref{ssec:proof72} to quasi-transitive graphs.
In contrast to transitive graphs, the faces are not 
generally regular polygons, and the longest edge need not be a diagonal.

Let $G\in\sQ$ be one-ended and not a triangulation.  
As before, we restrict ourselves to the case when $G$ is non-amenable, and
we embed $G$ canonically in the Poincar\'e disk $\sH$. The edges of $G$ are hyperbolic geodesics, 
but its diagonals need not be so. The  hyperbolic length of
an edge $e\in E_*\sm E$ does not generally equal the hyperbolic distance $\rho(e)$ between its endvertices.

The proof is an adaptation of that of Section \ref{ssec:proof72}, and full details are omitted. 
In identifying a path corresponding to the path $w$ of Lemma \ref{lem:t44}, we use the fact that edges of $E$ are 
geodesics, and concentrate on the \emph{final} departures of $\Ga_d^+$ from the faces whose interiors it enters.

\begin{remark}\label{rem:nonnece2}
The condition of Theorem \ref{t72qq} may be weakened as follows. In the above proof of Theorem \ref{thm:trans}
is constructed a \dinst\ of $G_*$ (see the discussion following Lemma \ref{lem:t44}). It suffices that, 
in the sense of that discussion, there exist
a diagonal $d$ and 
$s<0$, $t\ge 1$ such that (i) the path $(\nu_s, \nu_{s+1},\dots,\nu_t)$ is \nst\ in $G_*$, and
(ii)  for all $e\in E$ we have $p(\nu_t)-p(\nu_s) > p(\pi(e))$. Cf.\ Theorem \ref{thm:13}.
\end{remark}

\emph{Note added before publication}: the quasi-transitive case is treated in \cite{G24}.

\section*{Acknowledgements} 
ZL's research was supported by National Science Foundation grant 1608896 and Simons Collaboration Grant 638143.

\providecommand{\bysame}{\leavevmode\hbox to3em{\hrulefill}\thinspace}
\providecommand{\MR}{\relax\ifhmode\unskip\space\fi MR }
\providecommand{\MRhref}[2]{%
  \href{http://www.ams.org/mathscinet-getitem?mr=#1}{#2}
}
\providecommand{\href}[2]{#2}

\end{document}